\documentclass[11pt,twoside]{article}

\usepackage{a4}
\usepackage{amssymb,amsmath,amsthm,latexsym, mathrsfs}
\usepackage{amsfonts}
\usepackage{amsfonts}
\usepackage{graphicx}
\usepackage{amsmath, amsfonts}
\usepackage{amssymb, graphicx}
\usepackage{amscd}
\usepackage{textcomp}
\usepackage{palatino}
\usepackage{xcolor}
\usepackage{array}
\usepackage{multirow}
\usepackage{diagbox}
\usepackage[colorlinks=true,linkcolor=red,citecolor=red]{hyperref}

\newtheorem{theorem}{Theorem}[section]

\newtheorem{conjecture}{Conjecture}
\newtheorem{corollary}[theorem] {Corollary}

\newtheorem{lemma} [theorem]{Lemma}

\usepackage{indentfirst}
\allowdisplaybreaks
\vspace{5cm}

\begin{document}
\label{'ubf'}  
\setcounter{page}{1}
\markboth
{\hspace*{-9mm} \centerline{\footnotesize $q$-analogue of the generalized Stieltjes constants}}
{\centerline{\footnotesize Tapas Chatterjee and Sonam Garg } \hspace*{-9mm}}
\vspace*{-2cm}

\begin{center}
{{\textbf{On arithmetic nature of $q$-analogue of the generalized Stieltjes constants}}\\
\vspace{.2cm}
\medskip
{\sc Tapas Chatterjee \footnote{Research of the first author is partly supported by the core research grant CRG/2023/000804 of the Science and Engineering Research Board of DST, Government of India.}}\\
{\footnotesize  Department of Mathematics,}\\
{\footnotesize Indian Institute of Technology Ropar, Punjab, India.}\\
{\footnotesize e-mail: {\it tapasc@iitrpr.ac.in}}\\
\medskip
{\sc Sonam Garg\footnote{Research of the second author is supported by the University Grants Commission (UGC), India under File No.: 972/(CSIR-UGC NET JUNE 2018).}}\\
{\footnotesize Department of Mathematics, }\\
{\footnotesize Indian Institute of Technology Ropar, Punjab, India.}\\
{\footnotesize e-mail: {\it 2018maz0009@iitrpr.ac.in}}
\medskip}
\end{center}

\thispagestyle{empty} 
\vspace{-.4cm}
\hrulefill

\begin{abstract}  
{\footnotesize }
In this article, our aim is to extend the research conducted by Kurokawa and Wakayama in 2003, particularly focusing on the $q$-analogue of the Hurwitz zeta function. Our specific emphasis lies in exploring the coefficients in the Laurent series expansion of a $q$-analogue of the Hurwitz zeta function around $s=1$. We establish the closed-form expressions for the first two coefficients in the Laurent series of the $q$-Hurwitz zeta function. Additionally, utilizing the reflection formula for the digamma function and the identity of Bernoulli polynomials, we explore transcendence results related to $\gamma_0(q,x)$ for $q>1$ and $0 < x <1$, where $\gamma_0(q,x)$ is the constant term which appears in the Laurent series expansion of $q$-Hurwitz zeta function around $s=1$.

Furthermore, we put forth a conjecture about the linear independence of special values of $\gamma_0(q,x)$ along with $1$ at rational arguments with co-prime conditions, over the field of rational numbers. Finally, we show that at least one more than half of the numbers are linearly independent over the field of rationals.
\end{abstract}
\hrulefill



\vspace{-.37cm}

\section{\bf Introduction}

In 1882, Adolf Hurwitz introduced one of the generalizations of the Riemann zeta function namely, Hurwitz zeta function. The Riemann zeta function is absolutely convergent for all complex numbers $s$ and is defined by the following expression: $$\zeta(s) = \sum_{n=1}^{\infty}\frac{1}{n^s},$$ such that $\Re(s)>1$. Similarly, the Hurwitz zeta function is defined as:
\begin{align*}
\zeta(s,a) = \sum_{n=0}^{\infty}\frac{1}{(n+a)^s},
\end{align*}
where $s$ is a complex number with real part greater than $1$, and $0 < a \leq 1$. Like the Riemann zeta function, it can be analytically continued to the entire complex plane except at $s=1$ where it has a simple pole with residue one. This is reflected in the Laurent series expansion, as indicated: $$\zeta(s,a) = \frac{1}{s-1} + \displaystyle \sum_{n=0}^{\infty}\frac{(-1)^n}{n!}\gamma_n(a)(s-1)^n,$$ where $\gamma_n(a)$ are the generalized Stieltjes constants. The Hurwitz zeta function appears in many areas of mathematics, including number theory, complex analysis, and physics. It is particularly useful in studying the distribution of prime numbers and in the study of modular forms. The function has many interesting properties, including functional equations, integral representations, and connections to special functions such as the gamma function and polylogarithms. The Hurwitz zeta function is related to the gamma function, $\Gamma(s)$, with the following integral representation:
\begin{align*}
\zeta(s,a) = \frac{1}{\Gamma(s)}\int_{0}^{\infty}\frac{x^{s-1}e^{-ax}}{1-e^{-x}}dx,
\end{align*}
where $\Re(s) >1$ and $\Re(a)>0$. Since its introduction, it has been extensively studied by many mathematicians and has become an important tool in various fields of mathematics. Several mathematicians, including Landau, Ramaswami, and Wilton, have derived series representations for both the Riemann and the Hurwitz zeta functions in the past. Moreover, the Hurwitz zeta function satisfies several interesting functional equations, which relate its values at different points in the complex plane. These functional equations can be used to derive various identities and properties of the function. In other directions, many mathematicians have developed an increasing interest in an alternative generalization of classical functions referred to as $q$-analogues. A $q$-analogue refers to a modified version of a function or mathematical structure, usually involving a parameter $q$, such that as $q$ approaches $1$, the $q$-analogue reduces to the original function or structure, for any real number $0<q<1$ or $q>1$. A $q$-analogue is not unique and so among the numerous $q$-analogues of the Riemann zeta function, a notable study conducted by Kurokawa and Wakayama in 2003 focused on the following specific $q$-analogue of the Riemann zeta function (see ref.\,\cite{KW}):
\begin{align} \label{E1}
\zeta_q(s) = \sum_{n=1}^\infty\frac{q^n}{[n]_q^s}, ~~~~~~~~~ \Re(s) > 1
\end{align}
and studied a $q$-analogue of the Euler's constant. Further, Chatterjee and Garg in ref.\,\cite{TSG, TSG1, TSG2} extended their results related to $q$-analogue of the Riemann zeta function and the double zeta function and their algebraic identities. They also investigated the other coefficients in the Laurent series expansion of $q$-Riemann zeta function, $\gamma_k(q)$, and referred to them as a $q$-analogue of the Stieltjes constants. Also, they discussed the linear independence of the set of numbers involving $\gamma_0(q)$ over $\mathbb{Q}$. In addition to the above, Kurokawa and Wakayama also discussed a $q$-analogue of the Hurwitz zeta function in ref.\,\cite{KW}. Specifically, they investigated the following $q$-analogue of the Hurwitz zeta function, denoted as $\zeta_q(s,x)$, for the case when $q>1$:
\begin{align} \label{E2}
\zeta_q(s,x) = \sum_{n=0}^\infty\frac{q^{n+x}}{[n+x]_q^s}, ~~~~ \Re(s)>1,
\end{align}
where $x \notin \mathbb{Z}_{\leq 0}$ and gave a $q$-analogue of the limit formula of Lerch. In this regard, they gave the following theorem:
\begin{theorem}
Let $q>1$. Then, $\zeta_q(s,x)$ is meromorphic for $s \in \mathbb{C}$. Moreover, $\zeta_q(s,x)$ has a simple pole at $s=1$, and have the limit formula:
\begin{align*}
\lim_{s\rightarrow 1}\Bigg(\zeta_q(s,x) - \frac{q-1}{\log q} \cdot \frac{1}{s-1}\Bigg) = -\frac{q-1}{\log q} \cdot \frac{\Gamma_q^{\prime}}{\Gamma_q}(x).
\end{align*}
\end{theorem}

In this article, we further investigate a $q$-variant of the Hurwitz zeta function given by Equation (\ref{E2}). We begin our study with the following theorem:
\begin{theorem}\label{T1}
The $q$-analogue of the Hurwitz zeta function defined in Equation (\ref{E2}) is meromorphic for $s\in \mathbb{C}$ and its Laurent series expansion around $s=1$ is given by:
\begin{align*}
\zeta_q(s,x) = \frac{q-1}{\log q}.\frac{1}{s-1} + \gamma_0(q,x) + \gamma_1(q,x)(s-1) + \gamma_2(q,x)(s-1)^2 + \gamma_3(q,x)(s-1)^3 + \cdots
\end{align*}
with
\begin{align*}
\gamma_0(q,x)& = \sum_{n=1}^\infty\frac{q^{n(1-x)}}{[n]_q} + \frac{(q-1)\log(q-1)}{\log q} - \frac{q-1}{2} + (q-1)(1-x)
\end{align*}
and
\begin{align*}
\gamma_1(q,x)& = \Bigg(\sum_{n=1}^\infty\frac{q^{n(1-x)}}{[n]_q} + \frac{(q-1)\log(q-1)}{2\log q} - \frac{q-1}{2} + (q-1)(1-x)\Bigg)\log(q-1)\\
&\quad + \Bigg(\frac{q-1}{12} - \sum_{n=1}^\infty \frac{(1 + (q^n-1)x)q^{n(1-x)}}{[n]_q(q^n-1)} - \frac{(q-1)(1-x)x}{2}\Bigg)\log q\\
&\quad\quad + \sum_{n=1}^\infty\frac{q^{n(1-x)}s(n+1,2)}{n![n]_q},
\end{align*}
where $s(n+1,i)$ are the unsigned Stirling numbers of the first kind.
\end{theorem}

Now that we have presented the closed-form for $\gamma_0(q,x)$, we are ready to derive transcendence results pertaining to these constants. The statement of the result is outlined as follows:

\begin{theorem} \label{T2}
Let $q>1$ be any positive algebraic number, $b \geq 3$ be any integer, and $1 \leq a < b/2$ with $(a,b)=1$. Then,
\begin{align*}
\gamma_0 \left(q, \frac{a}{b} \right) - \gamma_0 \left(q,1- \frac{a}{b} \right) = \left( \frac{q-1}{\log q} \right) \pi\cot \left( \frac{\pi a}{b} \right)+ (2q-3) \left(\frac{1}{2} - \frac{a}{b} \right)
\end{align*}
is a transcendental number.
\end{theorem}

The theorem yields a noteworthy consequence, encapsulated in the following corollary:
\begin{corollary} \label{C1}
Let $q>1$ be any positive algebraic number and $b \geq 3$ be any integer. Then, 
\begin{align*}
\sum_{\substack{a=1 \\ (a , b) = 1}}^{\lfloor b/2 \rfloor}\left ( \gamma_0 \left(q, \frac{a}{b} \right) - \gamma_0 \left(q,1 - \frac{a}{b} \right) \right) & =  \left( \frac{q-1}{\log q} \right) \pi \sum_{\substack{a=1 \\ (a , b) = 1}}^{\lfloor b/2 \rfloor} \cot \left(\frac{\pi a}{b} \right)\\
&\quad + (2q-3)  \left(\frac{\varphi(b)}{4} - \frac{1}{b}\sum_{\substack{a=1 \\ (a , b) = 1}}^{\lfloor b/2 \rfloor} a \right),
\end{align*}
where $\varphi$ is the Euler's phi-function. Also, the above sum is a transcendental number.	
\end{corollary}
The focus of Theorem \ref{T2} is on the transcendental nature of the difference among the special values of $\gamma_0(q,x)$. So, one might ask questions regarding the interdependence between these numbers. Within this framework, we introduce the following theorem:
\begin{theorem} \label{T6}
	Let $q>1$ be any positive algebraic number and $b \geq 3$ be any integer. Then, the following set of numbers:
	$$\left \{	\gamma_0 \left(q, \frac{a}{b} \right) - \gamma_0 \left(q,1- \frac{a}{b} \right):1 \leq a < \frac{b}{2}, (a,b)=1\right \}$$
	is linearly independent over $\mathbb{Q}$.
\end{theorem}
As a consequence of Theorem \ref{T2} and \ref{T6}, we have the following result:
\begin{corollary} \label{C2}
	Let $q>1$ be any positive algebraic number and $b \geq 3$ be any integer. Then, the following set of numbers:
	$$ \left \{1, \gamma_0 \left(q, \frac{a}{b} \right) - \gamma_0 \left(q,1- \frac{a}{b} \right):1 \leq a < \frac{b}{2}, (a,b)=1 \right \}$$
	is linearly independent over $\mathbb{Q}$. In particular, a ratio of any two numbers in the above set is an irrational number.
\end{corollary}

On the other hand, one may ask questions concerning the arithmetic nature of $\gamma_0(q, a/b)$, where $q >1$ is a positive algebraic number, $b \geq 3$ is any integer, and $1 \leq a <b$ with $(a,b)=1$. In this regard, we formulate the following conjecture:
\begin{conjecture}
	Let $q>1$ be any positive algebraic number and $b \geq 3$ be any integer. Then, the following $\varphi(b) +1$ real numbers:
	$$ \left\{1, \gamma_0\left(q, \frac{a}{b}\right): 1 \leq a < b, (a,b)=1 \right\}$$
	are linearly independent over the field of rationals.	
\end{conjecture}

Let us define the vector space
$$V_{\mathbb{Q}}(q,b) = \mathbb{Q}-span~of \left\{1, \gamma_0\left(q, \frac{a}{b}\right): 1 \leq a < b, (a,b)=1 \right\}.$$
The above Conjecture A has an equivalent form which can be stated as follows:
\begin{conjecture}
	Let $q>1$ be any positive algebraic number and $b \geq 3$ be any integer. Then, 
	$$dim_{\mathbb{Q}}V_{\mathbb{Q}}(q,b) = \varphi(b) +1.$$
\end{conjecture}
Finally, we present the following theorem concerning the non-trivial lower bound of the dimension of the space, $V_{\mathbb{Q}}(q,b)$:
\begin{theorem} \label{T7}
	Let $q>1$ be any positive algebraic number and $b \geq 3$ be any integer. Then, at least 
	$$ \frac{\varphi(b)}{2} +1$$
	many numbers of the set $$\left\{1, \gamma_0\left(q, \frac{a}{b}\right): 1 \leq a < b, (a,b)=1 \right\}$$ are linearly independent over $\mathbb{Q}$.\\
	Equivalently, $$dim_{\mathbb{Q}}V_{\mathbb{Q}}(q,b) \geq \frac{\varphi(b)}{2} +1.$$
\end{theorem}

\section{\bf Notations and Preliminaries}
This section is dedicated to studying the notations and definitions relevant to the $q$-series, along with other essential results that will play a pivotal role in subsequent sections. \\
Let $a$ be a complex number. The $q$-analogue of $a$ is expressed by:\\
\begin{align*}
[a]_q = \frac{q^a - 1}{q - 1}, ~~~~~ q \neq 1.
\end{align*}
Additionally, the $q$-shifted factorial of $a$ is defined as:
\begin{align*}
(a;q)_0&=1, ~~~~~ (a;q)_n = \displaystyle\prod_{m=0}^{n-1} (1-aq^n), ~~~~~ n=1,2, \ldots\\
(a;q)_{\infty}& =\lim_{n\rightarrow\infty}(a;q)_n = \displaystyle\prod_{n=0}^{\infty} (1-aq^n).
\end{align*}
Furthermore, the $q$-analogue of the Lambert series is represented as:
\begin{align*}
\mathscr{L}_q(s,x) = \sum_{k=1}^{\infty}\frac{k^s q^{kx}}{1-q^k},~~ s\in \mathbb{C},
\end{align*}
where $|q| < 1$ and $x>0$.
Next, we explore some significant $q$-analogues of classical functions, commencing with the $q$-analogue of the gamma function introduced by Jackson \cite{J} as follows:
\begin{align*}
\Gamma_q(x) = \frac{(q;q)_{\infty} (1-q)^{1-x}}{(q^x;q)_{\infty}}, ~~~ \text{for} ~~ 0<q<1 
\end{align*}
and
\begin{align*}
\Gamma_q(x) = \frac{q^{\binom{x}{2}}(q^{-1};q^{-1})_{\infty} (q-1)^{1-x}}{(q^{-x};q^{-1})_{\infty}},~~~ \text{for} ~~ q>1.
\end{align*}
In the classical case, the logarithmic derivative of the classical gamma function is known as the digamma function. Likewise, a $q$-analogue of the digamma function is defined as the logarithmic derivative of a $q$-analogue of the gamma function. As a result, we have
\begin{align*}
\psi_q(x)  = \frac{d}{dx}\log \Gamma_q(x).
\end{align*}
Hence,
\begin{align*}
\psi_q(x)= -\log (1-q) + \log q \sum_{n \geq 0}\frac{q^{n+x}}{1-q^{n+x}}, ~~~~~ 0<q<1
\end{align*}
and
\begin{align}
\psi_q(x)& = -\log (q-1) + \log q\Bigg(x - \frac{1}{2} -  \sum_{n \geq 0}\frac{q^{-n-x}}{1-q^{-n-x}}\Bigg) \nonumber\\
& = -\log (q-1) + \log q\Bigg(x - \frac{1}{2} -  \sum_{n \geq 1}\frac{q^{-nx}}{1-q^{-n}}\Bigg), ~~~~~ q>1.\label{E3}
\end{align}
Further, we examine certain results that play a pivotal role in establishing our results. In 1970, S.\,Chowla \cite{SC} established the following theorem concerning the linear independence of cotangent values at rational arguments.
\begin{theorem}
	Let $p$ be a prime. Then, the $\frac{1}{2}(p-1)$ real numbers $\cot(\pi a/p)$, $ a = 1, \ldots ,\frac{1}{2}(p-1)$, are linearly independent over the field of rational numbers, $\mathbb{Q}$.
\end{theorem}

Then, in 1981, T.\,Okada \cite{TO} generalized Chowla’s theorem to encompass all derivatives of cotangent values. The formulation of his result is outlined below:

\begin{theorem} \label{T3}
	Let $k$ and $q$ be positive integers with $k \geq 1$ and $q > 2$. Let $T$ be a set of $\varphi(q)/2$ representations $\bmod~q$ such that the union $T \cup (-T)$ constitutes a complete set of co-prime residue classes $\bmod~q$. Then, the following set of real numbers: 
	\begin{align*}
		\frac{d^{k-1}}{dz^{k-1}} \cot (\pi z)|_{z = a/q},~~~~ a \in T
	\end{align*}
	is linearly independent over $\mathbb{Q}$.
\end{theorem}
In the classical context, Baker's theorem (see ref.\,\cite{AB}) assumes a crucial role in formulating assertions concerning the logarithms of algebraic numbers. This significance becomes apparent through the following statement:
\begin{theorem} \label{T4}
If $\alpha_1,\alpha_2,\ldots,\alpha_n$ are non-zero algebraic numbers such that $\log \alpha_1, \ldots, \log \alpha_n$ are linearly independent over the field of rational numbers, then $1, \log \alpha_1, \ldots, \log \alpha_n$ are linearly independent over the field of algebraic numbers.
\end{theorem}

In 2010, Murty and Saradha \cite{MS} (also see \cite{MM}) proved an important consequence of this result, which is given as follows:

\begin{lemma} \label{L1}
		Let $\alpha_1, \ldots, \alpha_n$ be positive algebraic numbers. If $c_0, c_1, \ldots, c_n$ are algebraic numbers with $c_0 \neq 0$, then $$c_0 \pi + \sum_{j=1}^n c_j \log \alpha_j$$ is a transcendental number and hence non-zero.
\end{lemma}

Further, in 2014, Chatterjee and Murty \cite{CM} recorded another variation of this lemma, which is stated as follows:

\begin{lemma} \label{L2}
		Let $\alpha_1, \ldots, \alpha_n$ be positive units in a number field of degree $> 1$. Let $r$ be a positive rational number unequal to $1$. If $c_0, c_1, \ldots, c_n$ are algebraic numbers with $c_0 \neq 0$ and $d$ is an integer, then $$c_0 \pi + \sum_{j=1}^n c_j \log \alpha_j + d \log r$$ is a transcendental number and hence non-zero.
\end{lemma}
We further state the modified version of Lemma \ref{L1}. The statement of the result is as follows:
\begin{lemma} \label{L3}
	Let $\alpha_1, \ldots, \alpha_n$ be positive algebraic numbers such that $\log \alpha_1, \ldots, \log \alpha_n$ are linearly independent over $\mathbb{Q}$. Then, the set of numbers
	$$ \{ 1, \pi, \log \alpha_1, \ldots, \log \alpha_n\} $$ is linearly independent over $\overline{\mathbb{Q}}$. In particular, $\frac{\pi}{\log \alpha_k}$ is transcendental for all $\alpha_k$.
\end{lemma}
\begin{proof}
First note that, we can represent $\pi = -i \log(-1)$. Additionally, in the light of Baker's theorem (see Theorem \ref{T4}), to establish the linear independence of the set of numbers $$\{ 1, \log(-1), \log \alpha_1, \ldots, \log \alpha_n \}$$  over $\overline{\mathbb{Q}}$, it is enough to show that $$\{ \log(-1), \log \alpha_1, \ldots, \log \alpha_n \}$$ is linearly independent over $\mathbb{Q}$.\\
Now, let integers $b_0, b_1, \ldots, b_n$ be such that
$$b_0 \log(-1) + b_1 \log \alpha_1 + \cdots + b_n \log \alpha_n = 0.$$
This implies $$\prod_{k=1}^n\alpha_k^{2b_k} = 1.$$
Since $\log \alpha_1, \ldots, \log \alpha_n$ are linearly independent over $\mathbb{Q}$, it follows that $\alpha_k$ for all $k \in \{1, \ldots, n\}$ are multiplicatively independent over $\mathbb{Q}$. Thus, $b_k = 0$ for all $k \in \{1, \ldots, n\}$ As a result, $b_0 = 0$. This completes the proof.
\end{proof}
Moreover, to outline the proof of Theorem \ref{T2}, we require a result given by Banerjee and Wilkerson in 2017 \cite{BW}. The statement of their theorem is expressed as follows:
\begin{theorem} \label{T5}
The Lambert series $\mathscr{L}_q(s, x)$ has the following expansion at $q=1$:
\begin{enumerate}
	\item For $s \neq 0,-1,-2, \ldots$,
	$$\mathscr{L}_q(s, x) \sim \frac{\Gamma(1+s) \zeta(1+s, x)}{t^{1+s}}+\sum_{k=0}^{\infty} \frac{(-1)^k \zeta(1-s-k) B_k(x)}{k !} t^{k-1}.$$
	
	\item For $s=0$,
	$$\mathscr{L}_q(0, x) \sim \frac{\psi(x)+\log \log \frac{1}{q} }{\log q}-\sum_{k=1}^{\infty} \frac{B_k B_k(x)}{k k !} (-\log q)^{k-1}.$$
	
	\item For $s=-m=-1,-2,-3, \ldots$,
	\begin{align*}
	\mathscr{L}_q(-m, x) \sim & \frac{(-1)^{m-1}\left[m \zeta^{\prime}(1-m, x)+\left(\log t-H_{m-1}\right) B_m(x)\right]}{m !} t^{m-1} \\
	&\quad +\sum_{k=0}^{m-1} \frac{(-1)^k \zeta(1+m-k) B_k(x)}{k !} t^{k-1} \\
	&\quad \quad +(-1)^{m-1} \sum_{k=m+1}^{\infty} \frac{B_{k-m} B_k(x)}{(k-m) k !} t^{k-1},
	\end{align*}
\end{enumerate}
where $B_k$ is the $k$-th Bernoulli number, $B_k(x)$ is the $k$-th Bernoulli polynomial, and $H_m$ is the $m$-th harmonic number.
\end{theorem}

\textbf{Note:} We shall consider $q>1$ throughout this paper.

\section{\bf Proofs of the main theorems}
\begin{proof}[\bf{Proof of Theorem \ref{T1}}]
The binomial expansion of a $q$-analogue of the Hurwitz zeta function results in
\begin{align}
\zeta_q(s,x)& =  (q-1)^s\displaystyle\sum_{n=0}^{\infty}q^{n+x}(q^{n+x}-1)^{-s} \nonumber \\
&=(q-1)^s\sum_{k=0}^{\infty}\frac{s(s+1)\cdots(s+k-1)}{k!}\frac{q^{(s+k-1)(1-x)}}{q^{s+k-1}-1}.\label{E4}
\end{align}
Then, clearly $\zeta_q(s,x)$ is meromorphic for $s \in \mathbb{C}$ and has simple poles at points in the set $\big\{1 +i\frac{2\pi b}{\log q} \mid b \in \mathbb{Z}\big\}~ \cup ~\big\{ a + i\frac{2\pi  b}{\log q} \mid a,b \in \mathbb{Z}, a \leq 0, b \neq 0\big\}$, with $s=1$ being a simple pole with residue $\frac{q-1}{\log q}$.
Now expanding Equation (\ref{E4}), we get:
\begin{align}
\zeta_q(s,x)=(q-1)^s\Bigg\{&\frac{q^{(s-1)(1-x)}}{q^{s-1}-1} + s \frac{q^{s(1-x)}}{q^s-1} + \frac{s(s+1)}{2}\frac{q^{(s+1)(1-x)}}{q^{s+1}-1} \nonumber\\
& + \frac{s(s+1)(s+2)}{6}\frac{q^{(s+2)(1-x)}}{q^{s+2}-1} + \cdots\Bigg\}.\label{E5}
\end{align}
Note that around $s=1$, we have:
\begin{align*}
(q-1)^s &= (q-1) + \{(q-1)\log(q-1)\}(s-1) + \frac{1}{2}\{(q-1)\log^2(q-1)\}(s-1)^2 \\
&\quad+ \frac{1}{6}\{(q-1)\log^3(q-1)\}(s-1)^3 + \cdots,\\
\frac{q^{(s-1)(1-x)}}{q^{s-1}-1}& = \frac{1}{\log q(s-1)} + \Bigg(\frac{1}{2} - x\Bigg) +\frac{1}{12}(1-6x + 6x^2)\log q(s-1)\\
&\quad+\frac{1}{12}(-x + 3x^2 -2x^3)\log^2 q(s-1)^2\\
& \quad \quad + \frac{1}{720}(-1 +30x^2- 60x^3 + 30x^4)\log^3 q (s-1)^3 + \cdots,\\
s \frac{q^{s(1-x)}}{q^s-1}& = \frac{q^{1-x}}{(q-1)} - \frac{q^{1-x}(1-q+ \log q -x \log q + qx \log q)}{(q-1)^2} (s - 1)\\
&\quad + \Bigg(\frac{(2- 2q -2x + 4qx - 2q^2x)}{2(q-1)^3}\Bigg)q^{1-x} \log q (s-1)^2\\
&\quad \quad +\frac{(1 + q -2x + 2qx +x^2 -2qx^2 + q^2x^2)}{2(q-1)^3}q^{1-x} \log^2 q (s-1)^2 +  \cdots.\\
\end{align*}
A similar expansion of the other terms in Equation (\ref{E5}) leads to the following expression:
\begin{align}
\zeta_q(s,x) & = \Bigg[(q-1) + \{(q-1)\log(q-1)\}{\bf{(s-1)}} + \frac{1}{2}\{(q-1)\log^2(q-1)\}{\bf{(s-1)^2}}\nonumber \\
&\quad+ \frac{1}{6}\{(q-1)\log^3(q-1)\}{\bf{(s-1)^3}} + \cdots\Bigg]\Bigg[\frac{1}{\log q}{\bf{\frac{1}{s-1}}} + \Bigg(\frac{1}{2} - x + \frac{q^{1-x}}{q-1} \nonumber\\
&\quad \quad + \frac{q^{2-2x}}{q^2-1} + \frac{q^{3-3x}}{q^3-1} + \cdots\Bigg){\bf{(s-1)^0}} +\Bigg( \frac{\log q}{12}(1-6x + 6x^2) \nonumber\\
& \quad\quad\quad - \frac{q^{1-x}(1-q+ \log q -x \log q + qx \log q)}{(q-1)^2}+ \cdots \Bigg){\bf{(s-1)}} + \cdots \nonumber\\
&= \frac{q-1}{\log q (s-1)} + \Bigg(\sum_{n=1}^{\infty}\frac{q^{n(1-x)}}{[n]_q} + \frac{(q-1) \log (q-1)}{\log q} - \frac{q-1}{2}  + (q-1)(1-x)\Bigg) \nonumber\\
& \quad +\Bigg(\sum_{n=1}^{\infty}\frac{q^{n(1-x)} \log(q-1)}{[n]_q} + \frac{(q-1) \log^2 (q-1)}{\log q} - \frac{q-1}{2} \log (q-1) \nonumber\\
&\quad \quad + (q-1)(1-x)\log(q-1) + \frac{q-1}{12}\log q - \sum_{n=1}^\infty \frac{(1 + (q^n-1)x)q^{n(1-x)}}{[n]_q(q^n-1)} \log q \nonumber\\
&\quad\quad\quad- \frac{(q-1)(1-x)x}{2}\log q + \sum_{n=1}^\infty\frac{q^{n(1-x)}s(n+1,2)}{n![n]_q}\Bigg) (s-1) + \cdots .\label{E6}
\end{align}
Therefore, we can conclude that
\begin{align}
\gamma_0(q,x)&= \sum_{n=1}^{\infty}\frac{q^{n(1-x)}}{[n]_q} + \frac{(q-1) \log (q-1)}{\log q} - \frac{q-1}{2}  + (q-1)(1-x), \label{E7}\\
\gamma_1(q,x)&= \Bigg(\sum_{n=1}^\infty\frac{q^{n(1-x)}}{[n]_q} + \frac{(q-1)\log(q-1)}{2\log q} - \frac{q-1}{2} + (q-1)(1-x)\Bigg)\log(q-1) \nonumber \\
&\quad + \Bigg(\frac{q-1}{12} - \sum_{n=1}^\infty \frac{(1 + (q^n-1)x)q^{n(1-x)}}{[n]_q(q^n-1)} - \frac{(q-1)(1-x)x}{2}\Bigg)\log q \nonumber\\
&\quad\quad + \sum_{n=1}^\infty\frac{q^{n(1-x)}s(n+1,2)}{n![n]_q}. \label{E8}
\end{align}
This completes the proof.
\end{proof}

\begin{proof}[\bf{Proof of Theorem \ref{T2}}]
From Equation (\ref{E7}), we have
\begin{align*}
\gamma_0(q,x)&= \sum_{n=1}^{\infty}\frac{q^{n(1-x)}}{[n]_q} + \frac{(q-1) \log (q-1)}{\log q} - \frac{q-1}{2}  + (q-1)(1-x)\\
& =  (q-1)\sum_{n=1}^{\infty}\frac{(1/q)^{nx}}{\left (1- \left(\frac{1}{q}\right)^n\right)} + \frac{(q-1) \log (q-1)}{\log q} - \frac{q-1}{2}  + (q-1)(1-x)\\
&= (q-1)\mathscr{L}_{1/q}(0,x) + \frac{(q-1) \log (q-1)}{\log q} - \frac{q-1}{2}  + (q-1)(1-x).
\end{align*} 
where $q>1$ and $ 0< x <1$. Now, consider
\begin{align*}
\gamma_0(q,x) - \gamma_0(q,1-x) =  (q-1)(\mathscr{L}_{1/q}(0,x) - \mathscr{L}_{1/q}(0,1-x)) + (q-1)(1-2x).
\end{align*} 
In light of Theorem \ref{T5}, when $s=0$, the expression for $\mathscr{L}q(0, x)$ is given by:
$$\mathscr{L}_q(0, x) = \frac{\psi(x)+\log \log \frac{1}{q} }{\log q}-\sum_{k=1}^{\infty} \frac{B_k B_k(x)}{k k !} (-\log q)^{k-1},$$
where $|q| <1$. So, for $q >1$, we have:
$$\mathscr{L}_{1/q}(0, x) = - \frac{\psi(x)+\log \log q }{\log q}-\sum_{k=1}^{\infty} \frac{B_k B_k(x)}{k k !} (\log q)^{k-1}.$$
Applying the reflection formula for the digamma function which is given as follows:
$$\psi(1-x) - \psi(x) = \pi \cot \pi x $$
and the relation of Bernoulli polynomials which is stated as follows:
$$B_k(1-x) = (-1)^k B_k(x),$$
we derive the following expression:
\begin{align*}
\gamma_0(q,x) - \gamma_0(q,1-x) = \left(\frac{q-1}{ \log q}\right)\pi \cot \pi x - 2 B_1 B_1(x) + (q-1)(1-2x),
\end{align*}
as $B_k$'s are zero for odd integer $k>1$.
Now, $$B_1 = -\frac{1}{2}$$
and $$B_1(x) = x - \frac{1}{2}.$$
Thus, 
$$\gamma_0(q,x) - \gamma_0(q,1-x) = \left(\frac{q-1}{ \log q}\right)\pi \cot \pi x + (2q-3)\left(\frac{1}{2} - x\right).$$
For $x = \frac{a}{b}$, where $b \geq 3$ and $ 1 \leq a < b/2$ with $(a, b) =1$, we obtain:
$$\gamma_0\left(q,\frac{a}{b}\right) - \gamma_0\left(q,1-\frac{a}{b}\right) = \left(\frac{q-1}{ \log q}\right)\pi \cot  \frac{\pi a}{b} + (2q-3)\left(\frac{1}{2} - \frac{a}{b}\right).$$
Finally, using Lemma \ref{L3}, we conclude that $\gamma_0\left(q,\frac{a}{b}\right) - \gamma_0\left(q,1-\frac{a}{b}\right)$ is a transcendental number, thereby concluding the proof.
\end{proof}

\begin{proof}[\bf{Proof of Corollary \ref{C1}}]
	From Theorem \ref{T2}, we have:
	$$\gamma_0\left(q,\frac{a}{b}\right) - \gamma_0\left(q,1-\frac{a}{b}\right) = \left(\frac{q-1}{ \log q}\right)\pi \cot \pi \frac{a}{b} + (2q-3)\left(\frac{1}{2} - \frac{a}{b}\right).$$
Now, by summing over all the $a$, where $1 \leq a < b/2$ with $(a,b) =1$, we get the desired result.
\end{proof}

\begin{proof}[\bf{Proof of Theorem \ref{T6}}]
	Let  $c_1, \ldots, c_{\varphi(b)/2} \in \mathbb{Q}$ be such that:
\begin{align*}
\sum_{\substack{1 \leq a_k < b/2 \\ (a_k , b) = 1}}c_k \left (\gamma_0\left(q,\frac{a_k}{b}\right) - \gamma_0\left(q,1-\frac{a_k}{b}\right)\right)  =0
\end{align*}
Substituting the value of $\gamma_0\left(q,\frac{a_k}{b}\right) - \gamma_0\left(q,1-\frac{a_k}{b}\right)$ for each $k$, we obtain:
\begin{align*}
&\left(\frac{q-1}{\log q}\right)\pi \left( c_1\cot \frac{\pi a_1}{b} + \cdots + c_{\varphi(b)/2}\cot \frac{\pi a_{\varphi(b)/2}}{b}\right)\\
&+ (2q-3) \left(c_1 \left(\frac{1}{2} - \frac{a_1}{b}\right) + \cdots + c_{\varphi(b)/2} \left(\frac{1}{2} - \frac{a_{\varphi(b)/2}}{b}\right)\right) = 0,
\end{align*}
which implies that
\begin{align*}
&\left(\frac{q-1}{\log q}\right)\pi \left( c_1\cot \frac{\pi a_1}{b} + \cdots + c_{\varphi(b)/2}\cot \frac{\pi a_{\varphi(b)/2}}{b}\right)\\
&= (3 - 2q) \left(c_1 \left(\frac{1}{2} - \frac{a_1}{b}\right) + \cdots + c_{\varphi(b)/2} \left(\frac{1}{2} - \frac{a_{\varphi(b)/2}}{b}\right)\right) 
\end{align*}
According to Lemma \ref{L3}, $\frac{\pi}{\log q}$ is a transcendental number. Therefore, the left-hand side of the above equality is an algebraic multiple of a transcendental number, while the right-hand side is an algebraic number. Hence, we must have:
$$  c_1\cot \frac{\pi a_1}{b} + \cdots + c_{\varphi(b)/2}\cot \frac{\pi a_{\varphi(b)/2}}{b}=0$$
and 
$$c_1 \left(\frac{1}{2} - \frac{a_1}{b}\right) + \cdots + c_{\varphi(b)/2} \left(\frac{1}{2} - \frac{a_{\varphi(b)/2}}{b}\right)=0.$$
Now, using Theorem \ref{T3} due to Okada for the case $k=1$, we know that the following set of numbers:
$$\left\{\cot \frac{\pi a_k}{b}: 1 \leq a_k < b/2, (a_k,b)=1, 1 \leq k \leq {\varphi(b)/2}\right\}$$
is linearly independent over $\mathbb{Q}$. Thus, $c_k = 0$ for all $ k \in \{1, \ldots, {\varphi(b)/2}\}$. Hence, the proof is complete.
\end{proof}

\begin{proof}[\bf{Proof of Corollary \ref{C2}}]
	The corollary is a direct consequence of Theorem \ref{T2} and Theorem \ref{T6}.
\end{proof}

\begin{proof}[\bf{Proof of Theorem \ref{T7}}]
First note that the space $V_{\mathbb{Q}}(q,b)$	 is also spanned by the following sets of real numbers:
\begin{align*}
 &\left \{1,	\gamma_0 \left(q, \frac{a}{b} \right) - \gamma_0 \left(q,1- \frac{a}{b} \right): 1 \leq a < \frac{b}{2}, (a,b)=1 \right \},\\
\text{and}~&\left \{	\gamma_0 \left(q, \frac{a}{b} \right) + \gamma_0 \left(q,1- \frac{a}{b} \right): 1 \leq a < \frac{b}{2}, (a,b)=1\right \}
\end{align*}
Now, from Theorem \ref{T2}, we obtain the expression:
\begin{align*}
\gamma_0 \left(q, \frac{a}{b} \right) - \gamma_0 \left(q,1- \frac{a}{b} \right) = \left( \frac{q-1}{\log q} \right) \pi \cot \left( \pi \frac{a}{b} \right) + (2q-3) \left(\frac{1}{2} - \frac{a}{b} \right)
\end{align*}
By employing Theorem \ref{T6}, we establish that the set
$$\left \{	\gamma_0 \left(q, \frac{a}{b} \right) - \gamma_0 \left(q,1- \frac{a}{b} \right):  1 \leq a < \frac{b}{2}, (a,b)=1\right \}$$
is linearly independent over $\mathbb{Q}$. Finally, utilizing Corollary \ref{C2}, we deduce that
$$dim_{\mathbb{Q}}V_{\mathbb{Q}}(q,b) \geq \frac{\varphi(b)}{2} +1.$$
\end{proof}


\end{document}